\documentclass[12pt]{amsart}

\usepackage{graphicx, overpic}
\usepackage[below]{placeins}
\usepackage[colorlinks=true, linkcolor=blue, citecolor=blue]{hyperref}
\usepackage[]{algorithm2e}
\usepackage[textsize=tiny]{todonotes}

\usepackage[T1]{fontenc}
\usepackage{amsmath,amsthm,amscd,amssymb,eucal}
\usepackage{enumerate, amsfonts, latexsym, color, url}
\usepackage{fancyvrb}
\CustomVerbatimCommand{\codestyle}{Verb}{formatcom=\ttfamily}
\usepackage{epstopdf}
\usepackage{pinlabel}
\usepackage{calrsfs}
\DeclareMathAlphabet{\pazocal}{OMS}{zplm}{m}{n}
\usepackage{tikz}
\usepackage{hyperref}

\DeclareMathOperator{\diam}{\mathrm{diam}}

\begin{document}

\newtheorem{theorem}{Theorem}[section]
\newtheorem*{theorem_A}{Theorem A}
\newtheorem*{theorem_B}{Theorem B}
\newtheorem{lemma}[theorem]{Lemma}
\newtheorem{corollary}[theorem]{Corollary}
\newtheorem{proposition}[theorem]{Proposition}
\newtheorem*{conjecture_main}{Conjecture}
\newtheorem{question}[theorem]{Question}

\theoremstyle{definition}
\newtheorem{definition}[theorem]{Definition}
\newtheorem{convention}[theorem]{Convention}
\newtheorem{notation}[theorem]{Notation}

\theoremstyle{remark}
\newtheorem{remark}[theorem]{Remark}
\newtheorem{example}[theorem]{Example}

\def\id{\text{id}}
\def\Z{\mathbb Z}
\def\N{\mathbb N}
\def\R{\mathbb R}
\def\C{\mathbb C}
\def\CC{\mathcal{C}}
\def\RR{\mathcal R}
\def\PP{\mathcal P}
\def\D{\mathbb D}
\def\Mod{\textnormal{Mod}}
\def\Diff{\textnormal{Diff}}
\def\length{\textnormal{length}}
\def\cost{\textnormal{cost}}
\def\inte{\textnormal{int}}
\def\myepsilon{\eta}

\title{Entropy and Translation Length in the Ray Graph}

\author{Juliette Bavard}
\address{Department of Mathematics \\ Univ Rennes \\ CNRS, IRMAR - UMR 6625, F-35000 Rennes, France}
\email{juliette.bavard@univ-rennes1.fr}
\author{Danny Calegari}
\address{Department of Mathematics \\ University of Chicago \\
Chicago, Illinois, 60637}
\email{dannyc@uchicago.edu}
\author{Alden Walker}
\address{Center for Communications Research \\ La Jolla, CA 92121, USA}
\email{akwalke@ccr-lajolla.org}
\date{\today}

\begin{abstract}
Let $\Gamma$ be the mapping class group of the plane minus a Cantor set, acting
on the ray graph $\RR$, and for $\gamma \in \Gamma$ let $\tau(\gamma)$ denote the
translation length of $\gamma$ on $\RR$.  We prove that
$\tau(\gamma) \le h(f)/\log{2}$ for every $C^\infty$ diffeomorphism $f$ of $S^2$
representing $\gamma$, where $h$ denotes topological entropy.  

The proof falls into two parts: a combinatorial argument bounding distance between
two rays in $\RR$ by the logarithm of the geometric intersection number; and a geometric
argument that promotes geometric control (coarse length of iterates of a fixed ray)
to combinatorial control (geometric intersection number).
We conjecture that the $C^\infty$ hypothesis on $f$ can be removed.
\end{abstract}

\maketitle

\section{Introduction}\label{section:introduction}

Let $K \subset \R^2$ be a Cantor set and let $\Gamma := \Mod(\R^2-K)$ be the
mapping class group of the plane minus $K$. The group $\Gamma$ is one of the best studied
examples of a (so-called) {\em big mapping class group}; see \cite{Aramayona_Vlamis} for an
overview of the subject.  

It is convenient to work in a compactified setting. Define $\bar{K}:=K \cup \infty \subset
S^2 = \R^2 \cup \infty$. Then $\Gamma$ may be identified with $\Mod(S^2,\bar{K})$, the
group of isotopy classes of orientation-preserving homeomorphisms of $S^2$ preserving $\bar{K}$
(note that since $\infty$ is the unique isolated point of $\bar{K}$, every homeomorphism
of $S^2$ permuting $\bar{K}$ must fix $\infty$, and therefore arises from a homeomorphism
of $\R^2$ preserving $K$).

The {\em ray graph} $\RR$ is the graph whose vertices are isotopy classes of
proper rays in $\R^2 - K$ from $\infty$ to a point of $K$, and whose edges are
pairs of classes admitting disjoint representatives (except possibly at endpoints in $K$).
This ray graph was first introduced in \cite{Calegari_blog}, and Bavard \cite{Bavard} 
showed that $\RR$ is Gromov hyperbolic of infinite diameter, and that some elements of $\Gamma$ act
on $\RR$ loxodromically. For $\gamma \in \Gamma$ the {\em translation length} of $\gamma$,
denoted $\tau(\gamma)$, is the limit
$$\tau(\gamma) := \lim_{n \to \infty} \frac{d_\RR(p,\gamma^n(p))}{n}$$
where $p \in \RR$ is an (arbitrarily chosen) basepoint. The element $\gamma$ is loxodromic
if and only if $\tau(\gamma)>0$.

For a homeomorphism $f$ of $S^2$ we write $h(f)$ for the topological
entropy of $f$, and for $\gamma \in \Gamma$ we define
$$h(\gamma) := \inf\left\{h(f)\;:\;f \text{ is a homeomorphism representing }\gamma\right\}$$
and
$$h_\infty(\gamma) := \inf\left\{h(f)\;:\; f \text{ is a } C^\infty \text{ diffeomorphism representing }\gamma\right\}$$
where in the second definition $f$ is required to preserve $\bar{K}$ but the smooth
structure on $S^2$ and the embedding of $K$ in it are arbitrary; since any two
Cantor sets in $S^2$ are ambiently homeomorphic this is no restriction.  Evidently
$h(\gamma) \le h_\infty(\gamma)$, and $h_\infty(\gamma) = \infty$ if and only if $\gamma$ has
no smooth representative.  Our main theorem is the following:

\begin{theorem_A}
Let $f$ be a $C^\infty$ diffeomorphism of $S^2$ preserving $\bar{K}$, and let
$\gamma$ be the class of $f$ in $\Gamma$.  Then
$$\tau(\gamma) \le \frac{h(f)}{\log{2}}$$
and consequently $\tau(\gamma) \le h_\infty(\gamma)/\log{2}$ for every
$\gamma \in \Gamma$.
\end{theorem_A} 

Inequalities of this form are familiar in the theory of mapping class groups of
surfaces of finite type.  For $S$ of finite type the group $\Mod(S)$ acts on the
curve complex $\CC(S)$ of Harvey \cite{Harvey}, which is hyperbolic
\cite{Masur_Minsky}, and $\gamma$ acts loxodromically if and only if it is
pseudo-Anosov, in which case $\log \lambda(\gamma)$, the logarithm of the stretch
factor, is the minimal entropy of a homeomorphism representing $\gamma$
\cite{Fathi_Shub, Handel}.  The relationship between $\tau(\gamma)$ and
$\log \lambda(\gamma)$ has been studied in some detail; in particular the minimal
value of the ratio $\log\lambda(\gamma)/\tau(\gamma)$ over pseudo-Anosov
$\gamma$ is comparable to $\log|\chi(S)|$ \cite{Gadre_Hironaka_Kent_Leininger,
Aougab_Taylor}, so that no inequality of the form of Theorem A can hold with a
constant independent of the topology of $S$.  It is therefore not obvious a priori
that a universal constant should exist at all for a surface of infinite type; the
content of Theorem A is that one does, and that it may be taken to be $1/\log 2$.

\subsection*{Smoothness}

It is natural to wonder whether $h_\infty(\gamma)$ can be replaced by $h(\gamma)$ in Theorem A.
The missing ingredient is an analogue of Thurston's classification theorem for
mapping classes that applies to surfaces of infinite type --- see
\cite{Bestvina_Fanoni_Tao} for a classification of ``tame'' homeomorphisms under
finiteness hypotheses on their accumulation sets --- and in particular an
analogue of the theorem of Fathi--Shub and Handel which identifies the minimal
entropy in a mapping class.  Its absence is the only reason that Theorem A is
stated for smooth representatives.  What is needed to remove the hypothesis is
the following statement, which we are unable to prove and which does not mention
the ray graph.  For a path $\gamma$ in a metric space
and $\epsilon>0$ let $V_\epsilon(\gamma)$, the {\em $\epsilon$-length} of
$\gamma$, be the largest $m$ for which the domain of $\gamma$ contains $m$
disjoint subintervals whose images have diameter at least $\epsilon$.

\begin{conjecture_main}
Let $f$ be a homeomorphism of a compact surface $S$, let $\alpha \subset S$ be
an arc, and let $\epsilon>0$.  Then
$$\limsup_{n \to \infty} \frac 1 n \log V_\epsilon\left(f^n(\alpha)\right)
\; \le \; h(f) .$$
\end{conjecture_main}

For a $C^\infty$ diffeomorphism the Conjecture is a theorem of Yomdin
\cite{Yomdin}, since $V_\epsilon(\gamma) \le \length(\gamma)/\epsilon$; this is
how smoothness enters the proof of Theorem A, and it is the only place where it
does. In particular:

\begin{theorem_B}
If the Conjecture holds then $\tau(\gamma) \le h(\gamma)/\log{2}$ for every
$\gamma \in \Gamma$.
\end{theorem_B}

\subsection*{Invariant Sets}

The main theorem and its proof hold verbatim if $K$ is {\em any} nonempty
compact totally disconnected subset of $\R^2$ (finite or infinite). In particular, 
we may take $K$ to depend on $f$. That is, if $f$ is any
$C^\infty$ diffeomorphism of $S^2$ fixing $\infty$, then $h(f)/\log{2} \ge \sup_K \tau_K(\gamma_K)$
where the supremum is taken over all $f$-invariant compact totally disconnected 
$K \subset \R^2$, where $\gamma_K$ is
the class of $f$ in the mapping class group $\Mod(S^2,\bar{K})$ for $\bar{K}:=K\cup \infty$,
and where $\tau_K$ is translation length of $\gamma_K$ in the ray graph of $K$.

In fact, one does not even need $K$ to be totally disconnected! Given any 
$f$-invariant compact set $K \subset \R^2$ we may replace each component $X$ of $K$ by its
full set --- i.e.\/ the minimal compact simply-connected subset $X' \subset \R^2$ containing
$X$ --- and take the quotient of $S^2$ by crushing each maximal full set (with respect 
to containment) to a point. By
Moore's theorem the quotient space is homeomorphic to $S^2$. The
homeomorphism $f$ descends to a homeomorphism $f'$ of this quotient $S^2$ fixing $\infty$ and leaving 
invariant a compact totally disconnected subset $K' \subset \R^2$ whose points are
the images of the (maximal) full components of $K$, and thereby a class $\gamma_{K'}$ in
$\Mod(S^2,\bar{K}')$. Note if $K$ is totally disconnected then of course $K'=K$. 
Because $S^2$ is compact, the quotient map $\pi:S^2 \to S^2$ is uniformly continuous
so for every arc $\gamma$ and every $\epsilon$ there is $\delta$ so that
$V_\epsilon(\pi \gamma) \le V_\delta(\gamma)$. Thus (following the notation of
the proof of Theorem~A in \S~\ref{section:proof}) we obtain the following chain of inequalities
$$\tau_{K'}(\gamma_{K'}) \le \frac {1} {\log 2} \limsup \frac{\log V_\epsilon((f')^n(\pi r^-))} {n} \le
\frac {1} {\log 2} \limsup \frac{\log V_\delta(f^n(r^-))} {n} \le \frac {h(f)}{\log 2}$$

In particular, we obtain the general inequality 
$h(f)/\log{2} \ge \sup_K \tau_{K'}(\gamma_{K'})$ where now $K$ ranges
over {\em all} $f$-invariant compact $K\subset \R^2$.

\subsection*{Outline}

The proof of Theorem A has two parts which are independent of each other, each of
which we think are of interest beyond the specific applications to this paper.

The first, carried out in \S\ref{section:divide_by_2}, is combinatorial: if two
rays cross $N$ times then their distance in the ray graph is at most
$\log_2 N + 2$ (Lemma \ref{lemma:divide_by_2} and Corollary \ref{cor:log_bound}).
This is the exact analogue for the ray graph of the classical bound
$d_{\CC(S)}(a,b) \le 2\log_2 i(a,b) + 2$ of Hempel \cite{Hempel}; 
the constant $1/\log 2$ of Theorem A comes from here.

The second, carried out in \S\ref{section:geometry}, is geometric: for a fixed
ray $r$ and a fixed metric on $S^2$ there is an $\epsilon > 0$ so that for any
ray $\gamma$ there is a ray $r'$, adjacent to $r$ in $\RR$, which crosses
$\gamma$ at most $3V_\epsilon(\gamma)$ times (Proposition \ref{prop:key_geometric}).
Here $\epsilon$ is a constant that may be derived from a fixed decomposition
of $S^2$ into a disk containing infinity, two disks containing $K$, and a complementary
pair of pants; the constant $\epsilon$ is chosen so that any arc crossing the pair 
of pants essentially has diameter at least $\epsilon$, so that
$\epsilon$-length controls the number of such crossings.

Combining the two parts bounds $d_\RR(r, f^n(r))$ by
$\log_2 V_\epsilon(f^n(r)) + O(1)$ for any homeomorphism $f$ representing
$\gamma$. Taking $n \to \infty$ reduces both theorems to the
Conjecture, and Yomdin's theorem proves it in the smooth case.  This is carried out in
\S\ref{section:proof}.

In \S\ref{section:hierarchy} we place the Conjecture in a hierarchy of possible
definitions of a stretch factor for an element of $\Gamma$, and record what
Theorem A proves about that hierarchy and what the Conjecture would add.  We also
give the reasons we believe the Conjecture, and sketch an approach to it for
loxodromic elements, via the classification of the Gromov boundary of $\RR$ in
\cite{Bavard_Walker} and an analogue of Handel's global shadowing theorem
\cite{Handel}.

\section{Rays and taut position}\label{section:taut_position}

Since all rays under consideration start at the common point $\infty$, it is
awkward to appeal to general position arguments. Thus we help ourselves to the
following convention in what follows

\begin{convention}\label{convention:disk}
Fix a closed round disk $B_\infty \subset S^2$ with
$\infty \in \inte(B_\infty)$ and disjoint from $K$, and set $\Sigma:=S^2 - \inte(B_\infty)$,
a closed disk containing $K$ in its interior. 

A {\em ray} is an embedded arc
$\alpha : [0,1] \to \Sigma$ with $\alpha(0) \in \partial\Sigma$,
$\alpha(1) \in K$ and $\alpha(0,1)$ contained in $\Sigma - K$.
The {\em equivalence class} of a ray is its isotopy class through rays.

Note that the endpoint of a ray on $\partial \Sigma$ may move freely during
an isotopy, whereas the endpoint on $K$ is fixed because $K$ is totally
disconnected.
\end{convention}

There is an obvious bijection between rays (and their equivalence classes) in the sense of 
\S~\ref{section:introduction} and rays (and their equivalence classes) 
in the sense of Convention~\ref{convention:disk} and we will move between these
two conventions without comment whenever it is unambiguous and the meaning is clear. 

\begin{definition}[Taut position]\label{def:taut}
Two rays $\alpha,\beta$ are in {\em general position} if they are transverse 
(except possibly at their common endpoints in $K$) and
meet in finitely many points of $\mathrm{int}(\Sigma)-K$.  An embedded disk
$B \subset \Sigma$ with $\inte(B)$ disjoint from $K$ is
\begin{enumerate}
\item{a {\em bigon} if $\partial B$ is the union of a subarc of $\alpha$ and a
subarc of $\beta$ meeting at two points of $\alpha \cap \beta$;}
\item{a {\em half-bigon at infinity} if $\partial B$ is the union of a subarc of
$\alpha$, a subarc of $\beta$, and a subarc of $\partial \Sigma$; or}
\item{a {\em half-bigon at $p$}, where $p = \alpha(1)=\beta(1)\in K$, if $\partial B$
is the union of a subarc of $\alpha$ and a subarc of $\beta$ meeting at $p$ and
at one point of $\alpha\cap\beta$.}
\end{enumerate}
Two rays $\alpha,\beta$ are in {\em taut position} if they are in general position
and there are no bigons or half-bigons.
\end{definition}

Given any two equivalence classes of rays $r_0,r_1$, 
we may easily find representatives $\alpha,\beta$ whose interiors are transverse. 
Now, it is possible that $|\alpha \cap \beta|=\infty$, and it is even possible that one cannot find
representatives of the rays with finite intersection. However, intersections of
interior transerse $\alpha \cap \beta$ can only accumulate at a (necessarily common) endpoint.
Thus one may always find a class $r_1$ and a representative $\beta'$ disjoint from $\beta$ (so that
$d_\RR(r_1,r_1')\le 1$) for which $\beta'(1) \ne \alpha(1)$, and then if we make the
interior of $\beta'$ transverse to $\alpha$, the resulting representatives will satisfy
$|\alpha \cap \beta'|<\infty$ and will be in general position in the sense of Definition~\ref{def:taut}.

The following lemma (and its proof) is standard:

\begin{lemma}[Tautening]\label{lemma:taut}
Let $\alpha,\beta$ be rays in general position with $|\alpha\cap\beta| = N$.
Then $\beta$ is isotopic to a ray $\beta'$ with $\alpha,\beta'$ in taut position
and $|\alpha \cap \beta'| \le N$.
\end{lemma}
\begin{proof}
Innermost bigons or half-bigons can be removed by an
isotopy without creating new intersections or ruining general position.
We remark that a half-bigon at infinity may be eliminated by an isotopy of
$\beta$ that moves the endpoint on $\partial \Sigma$.
\end{proof}

Another way to see this is to choose a complete hyperbolic structure on $\Sigma - K$ with
totally geodesic boundary and move the rays to geodesic representatives orthogonal
to $\partial \Sigma$.

\section{The Divide by 2 Lemma}\label{section:divide_by_2}

This section contains our first main result, the Divide by 2 Lemma, i.e.\/
Lemma~\ref{lemma:divide_by_2}. Given two rays $\alpha,\beta$ in
taut position we may construct a ray $\mu$ disjoint from $\alpha$ (except possibly
at a common endpoint) and with $|\mu \cap \beta| \le \lfloor N/2 \rfloor$.

\begin{lemma}[Divide by 2]\label{lemma:divide_by_2}
Let $\alpha,\beta$ be rays in taut position with $|\alpha\cap\beta| = N \ge 1$.
Then there is a ray $\mu$, disjoint from $\alpha$ except possibly at a common
endpoint in $K$, with $|\mu\cap\beta| \le \lfloor N/2\rfloor$.
Furthermore, if $\alpha$ and $\beta$ have distinct endpoints in $K$ then one
can arrange that
$|\mu\cap\beta| \le \lfloor (N-1)/2\rfloor$.
\end{lemma}

\begin{figure}[htbp]
\centering
\begin{tikzpicture}[scale=1.15]
\draw[thick] (0,0) circle (3);
\node at (18:3.32) {$\partial \Sigma$};
\node at (-2.30,2.30) {$\Sigma$};
\draw[very thick] (0,3) -- (0,-1.6);
\node[left] at (-0.10,2.74) {$\alpha$};
\draw[thick,red] plot [smooth,tension=0.7] coordinates
  {(2.60,1.50) (2.15,1.22) (1.75,0.98) (1.30,0.85) (0,0.90)};
\draw[thick,red] plot [smooth,tension=0.7] coordinates
  {(0,0.90) (-0.85,0.72) (-0.90,0.36) (0,0.20)};
\draw[thick,red] plot [smooth,tension=0.7] coordinates
  {(0,0.20) (1.60,-0.15) (2.25,-1.10) (1.90,-2.15) (0.60,-2.75) (-0.80,-2.70)
   (-2.05,-1.95) (-2.45,-0.60) (-2.30,0.90) (-1.70,1.95) (-0.85,2.35) (0,2.20)};
\draw[thick,red] plot [smooth,tension=0.7] coordinates
  {(0,2.20) (1.00,2.02) (1.05,1.72) (0,1.60)};
\draw[thick,red] plot [smooth,tension=0.7] coordinates
  {(0,1.60) (-1.30,1.25) (-1.80,0.20) (-1.70,-1.10) (-1.00,-1.95) (0.10,-2.25)
   (1.10,-1.85) (1.55,-1.00) (1.30,-0.55) (0,-0.50)};
\draw[thick,red] plot [smooth,tension=0.7] coordinates
  {(0,-0.50) (-0.95,-0.75) (-1.15,-1.45) (-0.55,-1.95) (0.35,-2.00) (0.90,-1.55)
   (0.72,-1.18) (0,-1.10)};
\draw[thick,red] plot [smooth,tension=0.7] coordinates
  {(0,-1.10) (-0.60,-1.35) (-0.40,-1.80) (0.30,-1.85)};
\fill[red] (2.60,1.50) circle (0.055);
\node[red] at (2.92,1.70) {$\beta$};
\foreach \y in {2.20,1.60,0.90,0.20,-0.50,-1.10} {\fill (0,\y) circle (0.055);}
\node at (-0.30,2.30) {$x_3$};
\node at (-0.30,1.50) {$x_4$};
\node at (0.28,1.07) {$x_1$};
\node at (-0.32,0.08) {$x_2$};
\node at (-0.34,-0.41) {$x_5$};
\node at (0.30,-0.98) {$x_6$};
\fill[blue] (0,-1.60) circle (0.075); \node[blue] at (-0.24,-1.68) {$p$};
\fill[blue] (0.30,-1.85) circle (0.075); \node[blue] at (0.56,-1.92) {$q$};
\fill[blue] (0.50,1.95) circle (0.06);
\fill[blue] (0.72,1.84) circle (0.06);
\fill[blue] (-0.45,0.61) circle (0.06);
\fill[blue] (-0.63,0.50) circle (0.06);
\fill[blue] (0.42,-1.64) circle (0.06);
\end{tikzpicture}
\caption{A configuration of rays $\alpha,\beta$ in taut position in $\Sigma$ with
$N=|\alpha \cap \beta| = 6$.  The crossings are labelled $x_1,\cdots,x_6$ in the
order in which they occur along $\beta$; along $\alpha$ they occur in the order
$x_3,x_4,x_1,x_2,x_5,x_6$.  Blue dots indicate points of $K$; the complementary
regions with no blue dot contain no points of $K$.}
\label{figure:rays}
\end{figure}
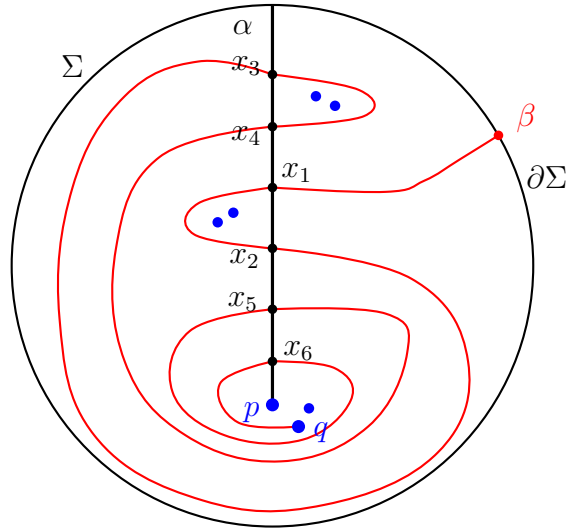
\begin{proof}
Suppose first that $p := \alpha(1)$ and $q := \beta(1)$ are distinct.  Let
$\Sigma_\alpha$ be the disk obtained by cutting $\Sigma$ along $\alpha$ and then
compactifying by gluing two copies $\alpha^\pm$ of $\alpha$ back in. Thus
$\partial \Sigma_\alpha$ decomposes into three arcs
$$\partial \Sigma_\alpha = \partial' \cup \alpha^+ \cup \alpha^-$$
and by abuse of notation we may write $p = \alpha^+ \cap \alpha^-$. We may think of
the Cantor set $K$ as a subset of $\Sigma_\alpha$, intersecting the boundary only at
the point $p$. The desired ray $\mu$ will be precisely a proper arc in $\Sigma_\alpha$
with initial point in the interior of $\partial'$, and terminal point in $K$ (the
case that the terminal point is $p$ is allowed).

Since the arc $\beta$ is in taut position, it is decomposed by $\alpha$ into finitely
many arcs $\beta_0,\cdots,\beta_N$ with endpoints $x_1,\cdots,x_N$ the points of
$\alpha \cap \beta$ in order along $\beta$ 
(recall that we are assuming that $\alpha(1)$ and $\beta(1)$ are distinct).
Each point $x_i$ determines a pair of points $x_i^\pm \in \alpha^\pm$ and by 
abuse of notation we may think of the $\beta_i$ as arcs in $\Sigma_\alpha$ so that
for $1\le i \le N-1$ the arc $\beta_i$ runs from one of $x_i^{\pm}$ to one of
$x_{i+1}^{\pm}$, and if $\beta_i$ ends at $x_{i+1}^+$ (say) then $\beta_{i+1}$ starts
at $x_{i+1}^-$ and vice versa. In particular, the arcs $\beta_0,\cdots,\beta_{N-1}$
are pairwise disjoint chords of $\Sigma_\alpha$. 

\begin{figure}[htbp]
\centering
\begin{tikzpicture}[scale=1.2]
\draw[thick] (0,0) circle (3);
\draw[line width=2.4pt,black!30] (210:3) arc (210:330:3);
\node[black!55] at (270:3.40) {$\partial'$};
\node at (32:3.70) {$\alpha^+$};
\node at (148:3.70) {$\alpha^-$};
\foreach \a in {330,-12,6,24,42,60,78,102,120,138,156,174,192,210,300}
   {\fill (\a:3) circle (0.055);}
\fill[blue] (90:3) circle (0.075);
\node at (330:3.48) {$\alpha^+(0)$};	
\node at (-12:3.32) {$x_3^+$};
\node at (6:3.30) {$x_4^+$};
\node at (24:3.30) {$x_1^+$};
\node at (42:3.30) {$x_2^+$};
\node at (60:3.30) {$x_5^+$};
\node at (78:3.32) {$x_6^+$};
\node[blue] at (90:3.28) {$p$};
\node at (102:3.32) {$x_6^-$};
\node at (120:3.30) {$x_5^-$};
\node at (138:3.30) {$x_2^-$};
\node at (156:3.30) {$x_1^-$};
\node at (174:3.32) {$x_4^-$};
\node at (192:3.32) {$x_3^-$};
\node at (210:3.48) {$\alpha^-(0)$};	
\node at (300:3.35) {$\beta(0)$}; 
\draw[thick,red] ([shift={(3.839,-1.247)}]-150.00:2.701) arc[start angle=-150.00, end angle=-246.00, radius=2.701];
\draw[thick,red] ([shift={(-2.547,1.654)}]-114.00:0.475) arc[start angle=-114.00, end angle=48.00, radius=0.475];
\draw[thick,red] ([shift={(-5.262,10.328)}]-48.00:11.196) arc[start angle=-48.00, end angle=-78.00, radius=11.196];
\draw[thick,red] ([shift={(3.033,-0.159)}]-102.00:0.475) arc[start angle=-102.00, end angle=-264.00, radius=0.475];
\draw[thick,red] ([shift={(-2.501,4.908)}]-96.00:4.620) arc[start angle=-96.00, end angle=-30.00, radius=4.620];
\draw[thick,red] ([shift={(-0.503,3.174)}]-150.00:1.152) arc[start angle=-150.00, end angle=-12.00, radius=1.152];
\draw[thick,red] (102:3) .. controls (99:2.72) .. (94:2.55);
\fill[blue] (94:2.55) circle (0.075); \node[blue] at (90:2.28) {$q$};
\node[red] at (338:1.68) {$\beta_0$};
\node[red] at (147:2.28) {$\beta_1$};
\node[red] at (137:0.62) {$\beta_2$};
\node[red] at (-3:2.32) {$\beta_3$};
\node[red] at (117:1.25) {$\beta_4$};
\node[red] at (99:1.82) {$\beta_5$};
\node[red] at (108:2.48) {$\beta_6$};
\fill[blue] (-8:2.82) circle (0.06);
\fill[blue] (2:2.82) circle (0.06);
\fill[blue] (143:2.80) circle (0.06);
\fill[blue] (151:2.80) circle (0.06);
\fill[blue] (86:2.76) circle (0.06);
\end{tikzpicture}
\caption{The disk $\Sigma_\alpha$ obtained by cutting $\Sigma$ along $\alpha$, for
the configuration of Figure~\ref{figure:rays}; the chords are drawn as hyperbolic
geodesics.  The boundary decomposes as $\partial' \cup \alpha^+ \cup \alpha^-$,
with $\partial'$ in grey, and the two copies of a point of $\alpha\cap\beta$ are
at equal angles.  The arcs $\beta_0,\cdots,\beta_5$ are chords; $\beta_6$ is the
free arc ending at $q$.}
\label{figure:cut}
\end{figure}
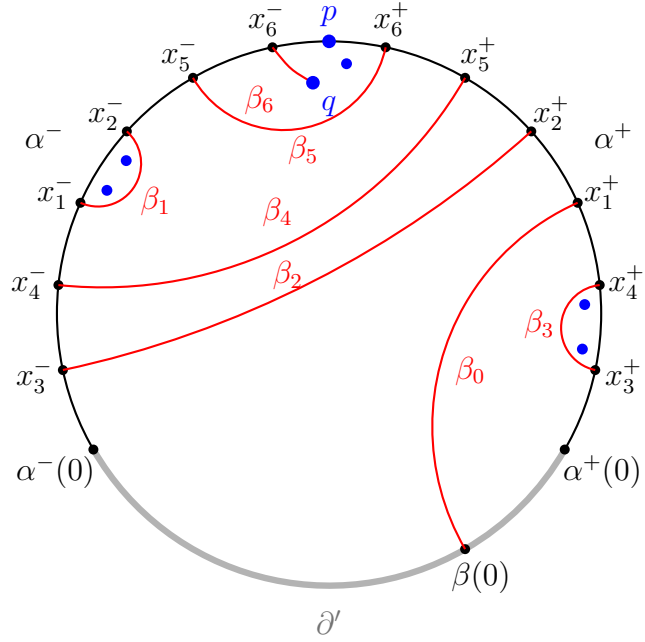

The $N$ chords cut $\Sigma_\alpha$ into $N+1$ regions, and the dual graph $T$ is a
tree because $\Sigma_\alpha$ is a disk. In particular, $T$ has $N+1$ vertices and 
$N$ edges. For a region $R$ of $\Sigma_\alpha - \cup \beta_i$ (i.e.\/ for a
vertex of $T$) let $\cost(R)$ be the least number of $\beta_i$ that must be crossed
by a path in $\Sigma_\alpha$ from $\partial'$ to $R$.

Only one $\beta_i$ has an endpoint on $\partial'$, namely $\beta_0$ which starts
at $\beta(0)\in \partial'$. Thus there are exactly two regions of $\Sigma_\alpha - \cup \beta_i$
with boundary on $\partial'$, the two sides of $\beta_0$. In other words, $\cost(R)$
is the distance in $T$ from $R$ to the edge (dual to) $\beta_0$.

The argument is concluded in three steps.

\medskip

\noindent{\bf Step 1: every leaf of $T$ contains an accessible point of $K$.} 
Suppose $R$ is a region dual to a leaf (i.e.\/ a 1-valent vertex) of $T$.
Let $\beta_k \subset \partial R$ be the chord dual to the edge with vertex dual to $R$.
We claim there is an embedded arc in $R$ from $\beta_k$ to some point of $K$
(possibly on $\partial R$) disjoint from $\beta_N$ except possibly at the end point. 

Since $R$ is a leaf, $\partial R = \beta_k \cup J$ where $J$ is an arc of $\partial \Sigma_\alpha$. 
There are two special cases: if $J$ contains $p$ we choose an arc from $\beta_k$ to $p$,
and if $R$ contains the free arc $\beta_N$ then we choose an arc from $\beta_k$ to 
$q$ (the endpoint of $\beta_N$) whose interior is disjoint from $\beta_N$.
Otherwise $R$ must contain interior points of $K$, or else $R$ would
either be a half-bigon at infinity, or a bigon (depending on whether $J$ contains
a subarc of $\partial'$ or not) which violates tautness, and we can simply choose
an arc from $\beta_k$ to some interior point of $K$.

\medskip

\noindent{\bf Step 2: some leaf has $\cost \le \lfloor(N-1)/2\rfloor$.}
Deleting the edge $\beta_0$ from $T$ leaves two subtrees $T_1,T_2$, containing the
two regions $R^1,R^2$ of cost $0$, with
$|E(T_1)| + |E(T_2)| = N-1$.  Choose $i$ with
$|E(T_i)| \le \lfloor (N-1)/2\rfloor$.  If $T_i$ has no edges then $R^i$ has
degree $1$ in $T$ and is itself a leaf, of cost $0$.  Otherwise let $R$ be any
leaf of $T_i$ other than $R^i$; then $R$ is a leaf of $T$, and the path in
$T_i$ from $R^i$ to $R$ uses at most $|E(T_i)|$ edges, so
$\cost(R) \le \lfloor (N-1)/2 \rfloor$.

\begin{figure}[htbp]
\centering
\begin{tikzpicture}[scale=1.2,
  vx/.style={circle,draw,fill=white,inner sep=0.5pt,minimum size=0.38cm},
  gd/.style={circle,draw,double,double distance=1.2pt,line width=0.5pt,fill=white,
             inner sep=0.5pt,minimum size=0.40cm}]
\draw[black!45] (0,0) circle (3);
\draw[line width=2.4pt,black!20] (210:3) arc (210:330:3);
\draw[black!40] ([shift={(3.839,-1.247)}]-150.00:2.701) arc[start angle=-150.00, end angle=-246.00, radius=2.701];
\draw[black!40] ([shift={(-2.547,1.654)}]-114.00:0.475) arc[start angle=-114.00, end angle=48.00, radius=0.475];
\draw[black!40] ([shift={(-5.262,10.328)}]-48.00:11.196) arc[start angle=-48.00, end angle=-78.00, radius=11.196];
\draw[black!40] ([shift={(3.033,-0.159)}]-102.00:0.475) arc[start angle=-102.00, end angle=-264.00, radius=0.475];
\draw[black!40] ([shift={(-2.501,4.908)}]-96.00:4.620) arc[start angle=-96.00, end angle=-30.00, radius=4.620];
\draw[black!40] ([shift={(-0.503,3.174)}]-150.00:1.152) arc[start angle=-150.00, end angle=-12.00, radius=1.152];
\draw[black!40] (102:3) .. controls (99:2.72) .. (94:2.55);
\coordinate (RI) at (335:2.30);
\coordinate (RII) at (-3:2.78);
\coordinate (RIII) at (245:1.90);
\coordinate (RIV) at (51:2.10);
\coordinate (RV) at (117:1.70);
\coordinate (RVI) at (86:2.62);
\coordinate (RVII) at (147:2.78);
\draw[very thick,densely dashed,red] (RI) -- (RIII);
\draw[thick,densely dashed] (RI) -- (RII);
\draw[thick,densely dashed] (RIII) -- (RIV);
\draw[thick,densely dashed] (RIV) -- (RV);
\draw[thick,densely dashed] (RV) -- (RVI);
\draw[thick,densely dashed] (RV) -- (RVII);
\node[red] at (0.72,-1.66) {$\beta_0$};
\node[vx] at (RI) {\small $0$};
\node[gd] at (RII) {\small $1$};
\node[vx] at (RIII) {\small $0$};
\node[vx] at (RIV) {\small $1$};
\node[vx] at (RV) {\small $2$};
\node[vx] at (RVI) {\small $3$};
\node[vx] at (RVII) {\small $3$};
\fill[blue] (3:2.85) circle (0.06);
\fill[blue] (143:2.80) circle (0.06);
\fill[blue] (151:2.80) circle (0.06);
\fill[blue] (94:2.55) circle (0.075); \node[blue] at (90:2.28) {$q$};
\fill[blue] (90:3) circle (0.075); \node[blue] at (90:3.28) {$p$};
\draw[very thick,blue,densely dotted] plot [smooth,tension=0.8] coordinates
  {(313:3) (325:2.30) (350:2.60) (3:2.85)};
\node[blue] at (332:1.98) {$\mu$};
\end{tikzpicture}
\caption{The dual tree $T$ of the decomposition of Figure~\ref{figure:cut}; each
vertex is labelled by its $\cost$.  Deleting the root edge $\beta_0$ (heavy
dashed) leaves subtrees with $1$ and $4$ edges, so Step 2 selects a leaf on the
smaller side: the doubled vertex, of $\cost = 1 \le \lfloor (N-1)/2\rfloor = 2$.
The ray $\mu$ (dotted) realises it, crossing $\beta$ once.  The two leaves of
$\cost = 3$ contain $q$, and $p$ on the boundary.}
\label{figure:tree}
\end{figure}

\medskip

\noindent{\bf Step 3: conclusion.}  Let $R$ be a leaf as in Step 2 and
let $\mu$ be a path in $\Sigma_\alpha$ from $\partial'$ to $R$ crossing
$\cost(R)$ chords transversally, followed by a path inside $R$ ending at the
point of $K$ supplied by Step 1.  Then $\mu$ is a ray, disjoint from $\alpha$ except possibly at $p$,
with $|\mu\cap\beta| \le \lfloor (N-1)/2\rfloor$.

\medskip

Finally suppose $\alpha(1) = \beta(1) = p$.  Then in $\Sigma_\alpha$ the last
piece $\beta_N$ runs from one of $x_N^{\pm}$ to the boundary point $p$, so it is a chord
and not a free arc; there are $N+1$ chords and $N+2$ regions,
$|E(T_1)|+|E(T_2)| = N$, and Step 2 produces a leaf of cost at most
$\lfloor N/2 \rfloor$.  Step 1 goes through with the obvious modification.
\end{proof}

\begin{corollary}\label{cor:log_bound}
Let $\alpha,\beta$ be rays in taut position with $|\alpha\cap\beta| = N \ge 1$.  Then
$$d_\RR(\alpha,\beta) \; \le \; \log_2{N}+2 .$$
\end{corollary}
\begin{proof}
Let $\mu$ be as in
Lemma \ref{lemma:divide_by_2} and tauten the pair $(\mu,\beta)$ using Lemma
\ref{lemma:taut}, which does not increase $|\mu\cap\beta| \le \lfloor
N/2\rfloor$. Now use $d_\RR(\alpha,\mu) \le 1$ and induction on $N$.
\end{proof}

\begin{remark}
Corollary \ref{cor:log_bound} is the ray graph analogue of Hempel's bound
$d_{\CC(S)}(a,b)\le 2\log_2 i(a,b)+2$ in the curve complex of a surface of
finite type \cite{Hempel}, with the constant $1$ in place of $2$; it is this
constant which becomes the $1/\log 2$ of Theorem A.  Estimates of this kind arise
inside proofs of hyperbolicity for graphs of curves and arcs, and in particular
inside the proof of hyperbolicity of $\RR$ in \cite{Bavard}, but we do not know a
reference for the statement in the form given here.
\end{remark}

\section{From dynamics to intersection numbers}\label{section:geometry}

\subsection{$\epsilon$-length}

\begin{definition}\label{def:epsilon_length}
Fix a metric $d$ on $S^2$.  For a path $\alpha:[0,1]\to S^2$ and $\epsilon>0$
the {\em $\epsilon$-length} of $\alpha$, denoted $V_\epsilon(\alpha)$, is the largest
$m$ for which there are disjoint subintervals
$I_1,\cdots,I_m \subset [0,1]$ with $\diam(\alpha(I_j))\ge\epsilon$ for every $j$.
\end{definition}

$V_\epsilon$ is a coarse substitute for length at a specific scale.

\begin{lemma}\label{lemma:V_finite}
$V_\epsilon(\alpha) < \infty$ for every path $\alpha$ and every $\epsilon>0$, and
$V_\epsilon(\alpha) \le \length(\alpha)/\epsilon$ if $\alpha$ is rectifiable.
\end{lemma}
\begin{proof}
By compactness of $I$ there is $\delta>0$ with
$d(\alpha(s),\alpha(t))<\epsilon$ whenever $|s-t|\le\delta$, proving the first statement. 
The second statement is immediate.
\end{proof}

\subsection{The three-disk model}

Let $\gamma \in \Gamma$ be a mapping class. Our goal is to obtain upper bounds on the translation length
of $\gamma$ in terms of the dynamics of any representative diffeomorphism $f$. The
estimates on the dynamical side will involve the $\epsilon$-length of iterates 
$f^n(\alpha)$ for some fixed ray $\alpha$. If there is a subsurface $S$ of $S^2 - \bar{K}$ 
containing $\infty$ and points of $K$ for which the rays $f^n(\alpha)$ have bounded topological complexity in $S$,
the orbit of $f^n(\alpha)$ in $\RR$ will have bounded diameter. Thus the only way for
$\tau(\gamma)$ to be big is for $f^n(\alpha)$ to become more and more topologically `entangled' 
in every such essential subsurface of $S^2 - \bar{K}$. For any compact essential subsurface $P$
of $S^2 - \bar{K}$, topologically essential arcs in $P$ have a definite diameter in $S^2$
(with respect to any fixed metric) and therefore $f^n(\alpha)$ is forced to have many subarcs
with definite diameter, which will be reflected in an exponential growth rate for
$V_\epsilon(f^n(\alpha))$. 

Since all the quantities we ultimately care about are invariant under a homeomorphism
of $S^2$, and since any two Cantor sets in $S^2$ are ambiently homeomorphic, we may fix
the following model once and for all. 

\begin{convention}\label{convention:pants}
Fix a metric on $S^2$ (for instance, the round
metric of constant curvature $1$) and fix three pairwise disjoint closed disks 
$A^+,A^-,B_\infty \subset S^2$ so that $\infty \in \inte(B_\infty)$ and $K$ is disjoint
from $B_\infty$ as in \S~\ref{section:taut_position}, and $K$ is written as the disjoint 
union of two Cantor sets $K^+$ and $K^-$ where each of $K^\pm$ is contained in $\inte(A^\pm)$.

Let $P$ be the complementary pair of pants, i.e.\/ 
$P:=S^2 - \inte(A^+) - \inte(A^-) - \inte(B_\infty)$.
Choose $\myepsilon>0$ so that every topologically essential proper embedded arc in $P$ has
diameter at least $\myepsilon$ in the given metric (the existence of such an $\myepsilon$ is
obvious).

Finally, fix two disjoint rays $r^\pm$ from $\infty$ to $K$ 
so that the intersection of each of the $r^\pm$
with $P$ is a single essential embedded arc, with $r^+$ running from $\partial B_\infty$
to $\partial A^+$ and $r^-$ running from $\partial B_\infty$ to $\partial A^-$.
\end{convention}

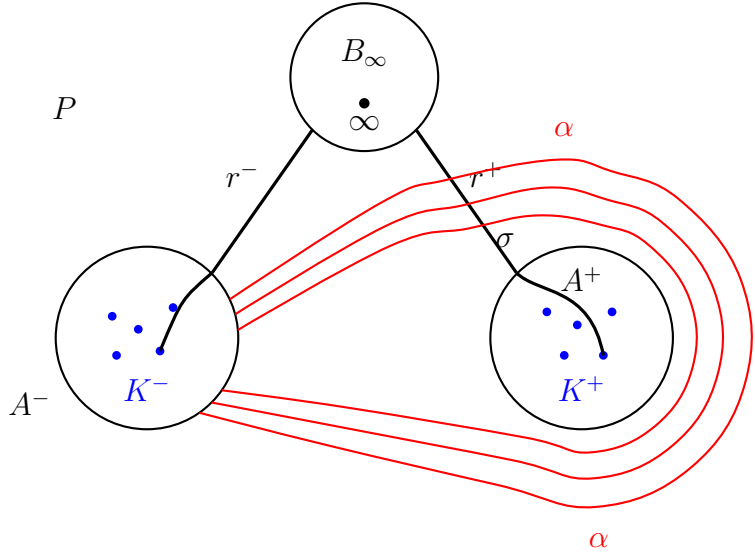
\begin{figure}[htbp]
\centering
\begin{tikzpicture}[scale=1.15]
\draw[thick] (0,1.9) circle (0.85);
\node at (0,2.18) {$B_\infty$};
\fill (0,1.60) circle (0.06); \node[below] at (0,1.58) {$\infty$};
\draw[thick] (-2.5,-1.1) circle (1.05);
\node at (-3.85,-1.85) {$A^-$};
\draw[thick] (2.5,-1.1) circle (1.05);
\node at (2.50,-0.42) {$A^+$};
\node at (-3.45,1.55) {$P$};
\foreach \x/\y in {-2.9/-0.85, -2.6/-1.0, -2.2/-0.75, -2.85/-1.3, -2.35/-1.25}
   {\fill[blue] (\x,\y) circle (0.05);}
\node[blue] at (-2.5,-1.68) {$K^-$};
\foreach \x/\y in {2.1/-0.8, 2.45/-0.95, 2.85/-0.8, 2.3/-1.3, 2.75/-1.3}
   {\fill[blue] (\x,\y) circle (0.05);}
\node[blue] at (2.5,-1.68) {$K^+$};
\draw[very thick] (-0.6,1.29) -- (-1.75,-0.35);
\draw[very thick] (-1.75,-0.35) .. controls (-2.1,-0.65) .. (-2.35,-1.25);
\node at (-1.40,0.78) {$r^-$};
\draw[very thick] (0.6,1.29) -- (1.75,-0.35);
\draw[very thick] (1.75,-0.35) .. controls (2.0,-0.6) and (2.6,-0.5) .. (2.75,-1.3);
\node at (1.40,0.78) {$r^+$};
\node at (1.62,0.02) {$\sigma$};
\draw[thick,red] plot [smooth,tension=0.7] coordinates
  {(-1.898,-1.960) (-0.400,-2.340) (1.525,-2.789) (2.670,-3.043)
   (3.753,-2.594) (4.384,-1.605) (4.332,-0.433) (3.618,0.497) (2.839,0.820)
   (2.200,0.950) (0.950,0.720) (0.200,0.460) (-1.548,-0.656)};
\draw[thick,red] plot [smooth,tension=0.7] coordinates
  {(-1.758,-1.842) (-0.300,-2.120) (1.690,-2.503) (2.641,-2.714)
   (3.541,-2.341) (4.065,-1.519) (4.022,-0.546) (3.429,0.227) (2.781,0.495)
   (2.100,0.630) (1.150,0.430) (0.250,0.190) (-1.486,-0.828)};
\draw[thick,red] plot [smooth,tension=0.7] coordinates
  {(-1.640,-1.702) (-0.200,-1.900) (1.840,-2.243) (2.615,-2.415)
   (3.348,-2.111) (3.775,-1.442) (3.740,-0.649) (3.257,-0.019) (2.729,0.200)
   (2.000,0.310) (1.200,0.160) (0.300,-0.040) (-1.454,-1.008)};
\node[red] at (2.70,-3.42) {$\alpha$};
\node[red] at (2.30,1.30) {$\alpha$};
\end{tikzpicture}
\caption{The three-disk model of Convention~\ref{convention:pants}, with
$\sigma = r^+ \cap P$.  The three components of $\alpha \cap P$ shown in red each
run from $\partial A^-$ to itself, separating $\partial B_\infty$ from
$\partial A^+$, and are therefore forced to cross $\sigma$.}
\label{figure:pants}
\end{figure}

\begin{definition}[Essential Complexity]
For $\alpha$ a ray from $\infty$ to $K$ let $e(\alpha)$ denote the number of components of 
$\alpha \cap P$ that are essential in $P$.
\end{definition}

\begin{lemma}\label{lemma:count_essential}
For any ray $\alpha$ we have the inequality $e(\alpha) \le V_\myepsilon(\alpha)$.
\end{lemma}
\begin{proof}
This follows from the definition of $\myepsilon$ and of $V_\myepsilon$. Notice that
although $\alpha \cap P$ may contain infinitely many components, only finitely many
of these will be essential.
\end{proof}

\begin{lemma}\label{lemma:pants_arcs}
Let $A$ be a finite collection of embedded essential arcs in $P$ and let $\nu$ be
a fixed embedded essential arc in $P$. Then we may freely isotop $A$ rel. $\partial P$ 
(i.e.\/ with endpoints free to move in $\partial P$) to a new embedded essential 
collection $A'$ so that each arc of $A'$ intersects $\nu$ transversely in at most two
points. 
\end{lemma}
\begin{proof}
There are only six free isotopy classes of essential arcs in $P$, one for each choice
of which components of $\partial P$ the endpoints live on. If we fix a hyperbolic
structure on $P$ with totally geodesic boundary, these six arcs are represented by 
orthogeodesics and the orthogeodesic representatives minimize the intersection numbers. 
Let's suppose $\nu$ itself is already an orthogeodesic. Then we may flow
components of $A$ by curve shortening, keeping them embedded and disjoint, until they 
converge to their orthogeodesic representatives; a tiny perturbation will give the
desired configuration.
\end{proof}

\begin{definition}
A ray $\alpha$ is {\em taut} with respect to $P$ if $\alpha$ is transverse to 
$\partial P$, if every component of $\alpha \cap P$ is essential in $P$, and if every
proper component of $\alpha \cap A^\pm$ is essential in $A^\pm - K$.
\end{definition}

\begin{proposition}\label{prop:normal_form}
Every ray $\alpha$ is isotopic to a ray $\alpha''$ taut with respect to $P$ of complexity
$e(\alpha'') \le V_\myepsilon(\alpha)$.
\end{proposition}
\begin{proof}
We may perturb $\alpha$ an arbitrarily small amount to a new ray $\alpha'$ which is
smooth and transverse to $\partial P$. Every innermost inessential arc of $\alpha' \cap P$
may be compressed in the complement of the essential arcs into $B_\infty$ or one of $A^\pm$. 
Likewise, every innermost inessential arc of $\alpha \cap A^\pm - K$ can be compressed
into $P$. Each compression reduces the number of intersections with $\partial P$.
Compressions of inessential arcs in $A^\pm - K$ connects up two arcs of $P$ (and possibly
makes them inessential). Thus this process terminates at some $\alpha''$ which has no
more essential arcs than $\alpha$. Thus $e(\alpha'') \le e(\alpha) \le V_\myepsilon(\alpha)$.
\end{proof}

\subsection{The key estimate}

\begin{proposition}\label{prop:key_geometric}
Let the ray $\alpha$ be taut with respect to $P$ and of complexity $e(\alpha)=e$.
Then there is a ray $\mu$ disjoint from $r^-$ and transverse to $\alpha$ with
$|\mu \cap \alpha| \le 3e$.

Consequently $d_\RR(r^-,\alpha) \le \log_2{3e} + 3$. 
\end{proposition}
\begin{proof}
We prove this in a sequence of steps.

\medskip

\noindent{\bf Step 1.} Let $\sigma\subset P$ be the arc $r^+\cap P$; this is an essential
arc, disjoint from $r^-$, running from $\partial B_\infty$ to $\partial A^+$. The
first step is to freely isotop $\alpha$ (by an isotopy supported
in a thin collar neighborhood of $P$ disjoint from $K$) so that $\alpha \cap P$ 
intersects $\sigma$ transversely in at most $2e$ points. This can be done
by Lemma~\ref{lemma:pants_arcs}.

\medskip

\noindent{\bf Step 2.} After the isotopy in Step 1, the 
intersection of $\alpha$ with $A^+$ is a system of at most
$e+1$ disjoint compact arcs, at most one of which ends on
a point of $K$, and such that all of the others are essential in $A^+-K$. 
We claim that we can choose an embedded path 
$\sigma'$ in $A^+$ from some initial point $\sigma'(0)\in \partial A^+$ to some point 
$\sigma'(1) \in K$, and so that $\sigma'$ is disjoint from $\alpha$ except possible at 
$\sigma'(1)$. To see this, choose an innermost bigon region and observe that it must
contain some points of $K$, or else some arc would be inessential.

\medskip

\noindent{\bf Step 3.} We may join $\sigma(1)$ to $\sigma'(0)$ by an arc of
$\partial A^+$ that intersects at most half of the points of $\alpha \cap \partial A^+$,
i.e.\/ at most $e$ points. This union can be extended in $B_\infty$ to a proper
ray $\mu$, disjoint from $r^-$, and transverse to $\alpha$ with $|\mu \cap \alpha| \le 3e$.
Now apply Corollary \ref{cor:log_bound}, after tautening the pair $(\mu,\alpha)$ using Lemma
\ref{lemma:taut}.
\end{proof}

\section{Proof of the theorems}\label{section:proof}

Putting this together proves the main theorems:

\begin{proof}
Let $f$ be a diffeomorphism of $S^2$ representing $\gamma$ and choose $r^-$ smooth
as in Convention~\ref{convention:pants}. By Proposition~\ref{prop:normal_form}
each $f^n(r^-)$ is isotopic to a ray taut with respect to $P$ of complexity $e(f^n(r^-))$,
so by Proposition~\ref{prop:key_geometric} we have
$$d_\RR(r^-,\gamma^n(r^-)) \le \log_2{3e(f^n(r^-))} +3$$
and $e(f^n(r^-)) \le V_{\myepsilon}(f^n(r^-))$ by Lemma~\ref{lemma:count_essential}. Hence
$$\tau(\gamma) \le \frac {1} {\log 2} \limsup_{n \to \infty} \frac {\log V_\myepsilon(f^n(r^-))} {n} \le \frac {h(f)}{\log 2}$$
where the last inequality follows by Yomdin's theorem \cite{Yomdin} (see also \cite{Gromov}).
This proves Theorem A.

If $f$ is merely a homeomorphism then the last inequality is the statement of our main
conjecture, and Theorem B follows in the same way.
\end{proof}

\begin{remark}
The $C^\infty$ hypothesis in Yomdin's theorem cannot be weakened to $C^r$, by
examples of Misiurewicz; the reverse inequality, that entropy is at most the
volume growth rate, is due to Newhouse \cite{Newhouse} and needs only
$C^{1+\epsilon}$.  So for $C^\infty$ diffeomorphisms entropy and volume growth
agree, and $h(f)$ may be replaced in Theorem A by the length growth rate of a
single arc.
\end{remark}

\section{Stretch factors}\label{section:hierarchy}

\subsection{A hierarchy}

Fix a metric on $S^2$ and let $\gamma \in \Gamma$.  For a path $\alpha$ and a
homeomorphism $f$ put
$$\sigma_\epsilon(f,\alpha) := \limsup_{n\to\infty}\frac{\log V_\epsilon(f^n(\alpha))}{n},
\qquad\quad \sigma(f,\alpha) := \lim_{\epsilon\to 0}\sigma_\epsilon(f,\alpha),$$
and let $\sigma^{\mathrm{iso}}(f,\alpha)$ be defined in the same way with
$V_\epsilon(f^n(\alpha))$ replaced by the infimum of $V_\epsilon$ over rays
isotopic to $f^n(\alpha)$.  Consider the following quantities.
\begin{enumerate}
\item $h_\infty(\gamma)$, the infimum of $h(f)$ over $C^\infty$ representatives;
\item $h(\gamma)$, the infimum of $h(f)$ over all representatives;
\item the infimum over representatives $f$ of the supremum of
$\sigma(f,\alpha)$ over compact arcs $\alpha$;
\item the same, with the supremum restricted to rays;
\item the infimum over representatives of the {\em infimum} of $\sigma(f,R)$
over rays $R$;
\item the same with $\sigma$ replaced by $\sigma^{\mathrm{iso}}$;
\item $\log 2 \cdot \tau(\gamma)$.
\end{enumerate}
Trivially $(1)\ge(2)$ and $(3)\ge(4)\ge(5)\ge(6)$.  Yomdin's theorem gives
$(1)\ge(3)$, and \S\S\ref{section:divide_by_2}--\ref{section:proof} give
$(6)\ge(7)$; these are the two inputs to Theorem A, which is the composite
$(1)\ge(3)\ge\cdots\ge(7)$.  The Conjecture is exactly the assertion
$(2)\ge(3)$, and Theorem B is the resulting composite
$(2)\ge(3)\ge\cdots\ge(7)$.

So Theorem A says that no quantity in the hierarchy from (1) down can be smaller
than $\log 2 \cdot \tau(\gamma)$, while the Conjecture would insert $h(\gamma)$
into the chain at the top, which is what the sharp form of the inequality
requires.  We do not know whether any of the inequalities above is an equality; in
particular we do not know whether $h(\gamma) = h_\infty(\gamma)$, that is, whether
the infimal entropy in a mapping class is approximated by smooth representatives
preserving a smoothly embedded Cantor set.  A positive answer would give the sharp
inequality without the Conjecture.

\subsection{Evidence for an unconditional Theorem B}

One route to prove an unconditional Theorem B would be to prove our main Conjecture.
As we have explained, this Conjecture is true for $C^\infty$ diffeomorphisms by
Yomdin \cite{Yomdin}, and there are
partial results in intermediate regularity \cite{Zang}. Examples due to Misiurewicz
which invalidate Yomdin's theorem with regularity below $C^\infty$ are irrelevant
to the conjecture, since they inflate the length of arcs by inserting tiny
wiggles that are invisible at the level of $V_\epsilon$. So we think there is a 
decent chance the conjecture is true.

Independently of the truth of this conjecture, there are routes to try to establish
an unconditional Theorem B. For surfaces of finite type, a result of this
kind is known (as we already remarked): if $g$ is a homeomorphism of a compact surface $S$ with a
finite invariant set $F$, then the exponential growth rate of the geometric
intersection numbers $i(g^n(\alpha),\beta)$ of any two curves $\alpha$, $\beta$ in $S - F$ 
is at least $\log\lambda$ where $\lambda$ is the stretch factor for the Thurston representative 
of the class, so that $h(g)\ge\log\lambda$ (\cite{Fathi_Shub, Handel}). Therefore one is led to look
for an analogue of Thurston's classification in the infinite type setting 
(see \cite{Bestvina_Fanoni_Tao} for partial results in this direction) and
a monotone comparison theorem for entropy.

Since Theorem B is vacuously true for elements which are not loxodromic, let's
focus attention on loxodromic $\gamma$. Suppose $f$ is a representative homeomorphism
of $\gamma$. A lower bound on the entropy of $f$ may be obtained by exhibiting 
$f$-invariant horseshoes, i.e.\/ rectangles together with a family of horizontal
strips whose $f$-images are vertical strips crossing the horizontal strips
transversely. 

If $\gamma$ is loxodromic, after passing to a power if necessary, it fixes two
points in the Gromov boundary $\partial_\infty\RR$ which (by \cite{Bavard_Walker})
are represented by attracting and repelling {\em high filling long
rays} $R^+$ and $R^-$. One hopes that subarcs of $R^+$ and $R^-$ will form sides of
rectangles as above and therefore certify a lower bound for entropy for any
representative $g$ of $\gamma$ preserving $R^+$ and $R^-$.

To compare the entropy of $g$ to that of an arbitrary representative $f$ one needs
an infinite type analog of Handel's global shadowing theorem \cite{Handel}.
Here one might hope to adapt Handel's fixed point theorem for planar homeomorphisms
\cite{Handel_fixed_point}, whose hypotheses involve only finitely many orbits and
their cyclic order at infinity, data which the pair $R^{\pm}$ supplies. Another
potential tool is the forcing theory of Le Calvez and Tal
\cite{Le_Calvez_Tal}, which produces topological horseshoes for $C^0$ surface
homeomorphisms out of the combinatorial complexity of transverse trajectories.

\section{Acknowledgements}

We would like to thank Autumn Kent and Saul Schleimer for useful conversations
about this material. The figures in this paper and some minor editing and 
proofreading of the prose were done with the assistance of Claude (Anthropic) 
but the authors take full intellectual responsibility for the content.

\sloppy

\end{document}